\documentclass[a4paper,12pt]{elsarticle}
\usepackage{amscd,amsmath,amsthm,amsfonts}
\usepackage{graphicx}
\usepackage{amssymb,color,amsbsy}

\usepackage[matrix,arrow]{xy}
\usepackage{mathabx}

\usepackage{tikz}
\usetikzlibrary{decorations.pathmorphing}
\usepackage{caption}
\usepackage{subcaption}

\usepackage[pdftex]{hyperref}
\hypersetup{colorlinks=false}
\def\theenumi{\arabic{enumi}}

\def\theenumii{\alph{enumii}}
\def\p@enumii{\theenumi.}

\def\theenumiii{\arabic{enumiii}}
\def\p@enumiii{(\theenumi)(\theenumii)}

\def\p@enumiv{\p@enumiii.\theenumiii}
\newtheorem{theorem}{Theorem}[section]

\newtheorem{algorithm}[theorem]{Algorithm}

\newtheorem{conjecture}[theorem]{Conjecture}
\newtheorem{corollary}[theorem]{Corollary}

\newtheorem{definition}[theorem]{Definition}

\newtheorem{lemma}[theorem]{Lemma}

\newtheorem{proposition}[theorem]{Proposition}
\newtheorem{remark}[theorem]{Remark}

\DeclareMathAlphabet{\mathpzc}{OT1}{pzc}{m}{it}

\newcommand{\nulidad}[1]{\operatorname{null}(#1)}
\newcommand{\rank}[1]{\operatorname{rk}(#1)}

\newcommand{\nulidadm}[1]{\operatorname{null_m}(#1)}
\newcommand{\nulidadM}[1]{\operatorname{null_M}(#1)}

\newcommand{\mna}[1]{\operatorname{mna}{(#1)}}
\newcommand{\MNA}[1]{\operatorname{MNA}{(#1)}}

\tikzstyle{every node}=[circle, draw, fill=white!,
inner sep=0.1pt, minimum width=16pt]

\journal{\ldots}

\begin{document}


\begin{frontmatter}
	\title{Maximum and minimum nullity of a tree degree sequence}
	
	\author[daj]{Daniel A Jaume\corref{cor1}}
	\ead{daniel.jaume.tag@gmail.com}
	
	\author[daj]{Gonzalo Molina}
	\ead{gonzalo.molina.tag@gmail.com}
	
	\cortext[cor1]{Corresponding author: Daniel A. Jaume}
	
	\address[daj]{	Departamento de Matem\'{a}ticas.
		Facultad de Ciencias F\'{\i}sico-Matem\'{a}ticas y Naturales. Universidad Nacional de San Luis.
		1er. piso, Bloque II, Oficina 54. Ej\'{e}rcito de los Andes 950.
		San Luis, Rep\'{u}blica Argentina.
		D5700HHW.}
	
	
	\begin{abstract}
The nullity of a graph is the multiplicity of the eigenvalue zero in its adjacency spectrum. In this paper, we give a closed formula for the minimum and maximum nullity among trees with the same degree sequence, using the notion of matching number and annihilation number. Algorithms for constructing such minimum-nullity and maximum-nullity trees are described.
	\end{abstract}
	
	
\end{frontmatter}


\section{Introduction}

Collatz and Sinogowitz (1957), see \cite{von1957spektren}, first raised the problem of characterizing all singular or nonsingular graphs. This problem is hard to be solved. On one hand, the nullity is relevant to the rank of symmetric matrices described by graphs. On the other hand, the nullity has strong chemical background. A singular bipartite graph expresses the
chemical instability of the molecule corresponding to the bipartite graph. Due to all these reasons, the nullity aroused the interest of many mathematicians and chemists. The topics on the nullity of graphs includes the computing nullity, the nullity distribution, bounds
on nullity, characterization of graphs with certain nullity, and so on. The nullity of trees has been study in many works, for example \cite{fiorini2005trees}, \cite{li2006trees}, \cite{ghorbani2016integral}, and \cite{jaume2018null}.

In \cite{fiorini2005trees}, Fiorini, Gutman and Sciriha determined the greatest nullity among $n$-vertex trees in which no vertex has degree greater than a fixed value \(D\). They explicit constructing the respective trees. Our work can be seen as an extension of Fiorini, Gutman, and Sciriha's results. Let $\mathbf{s}:d_1,d_2,\ldots,d_n$ be a degree sequence of length $n$ such that $d_i\leq d_{i+1}$ for $1\leq i \leq n$.  We say that \(\mathbf{s}\) is a degree sequence of length \(n\).  A degree sequence \(\mathbf{s}\) is a \textbf{tree degree sequence} if \(\sum_{i=1}^{n} d_{i}=2n-2\), see \cite{chartrand2010graphs}. Note that if a graph has tree degree sequence \(\mathbf{s}\) and it is disconnected, one of its connected components can be an unicycle graph. Let \(\mathcal{G}_\mathbf{s}\) be the set of all graphs with degree sequence \(\mathbf{s}\). With $\mathcal{T}_\mathbf{s}$ we denote the set of all the connected graph in \(\mathcal{G}_\mathbf{s}\). Let \(\mathbf{s}\) be a tree degree sequence, $T \in \mathcal{T}_\mathbf{s}$ if and only if \(T\) is a tree and has degree sequence $\mathbf{s}$.

Finally, some comments about the notation. All graphs in this work are labeled (even when we do not write the labels), finite, undirected and with neither loops nor multiple edges. Let \(G\) be  a graph, where $V(G)$ and $E(G)$ denotes the set of vertices and the set of edges of $G$, respectively. Let \(v \in V(G)\), the neighborhood of \(v\), denoted by \(N(v)\), is the set \(\{u \in V(G) \, : \, u \sim v\}\). The neighborhood of a subset \(S\) of \(V(G)\) is
\[
N(S):= \bigcup_{v \in S} N(v).
\]  
With \(\deg(v)\), we denote the degree of a vertex \(v\) of a graph \(G\), that is, \(\deg(v)=|N(v)|\). A vertex \(v\) of a graph \(G\) is a a leaf if \(\deg(v)=1\). With \(l(G)\) we denote the set of all the leaves of \(G\), i.e. \(l(G)=\{v \in V(G) : \deg v =1\}\). If the graph \(G\) is clear from the context, then we just write \(l\). The vertices of \(G\) that are not leaves are called \textbf{internal} vertices of \(G\). Thus, if \(v\) is an internal vertex of \(G\), then \(\deg v >1\). Let \(u,v \in V(G)\), with \(G+uv\) we denote the graph obtained by adding up the edge \(uv\) to \(E(G)\). Let $G_1$ and $G_2$ be two graphs, with $G_1\cup G_2$ we denote the union of two graphs, where $V(G_1\cup G_2)=V(G_1)\cup V(G_2)$ and $E(G_1\cup G_2)=E(G_1)\cup E(G_2)$. 

Regarding algebraic notions, the only concept that is relevant to this work is the nullity of a graph, that is, the multiplicity of eigenvalue zero in the spectrum of the adjacency matrix of a graph. For all graph-theoretic and algebra-theoretic notions not defined here, the reader is referred to~\cite{chartrand2010graphs} and \cite{meyer2000matrix}, respectively.

The following two well-known results are crucial to this work.
\begin{theorem}[\cite{D1972} and \cite{bevis1995ranks}]\label{bevis}
	For any tree $T$, $\rank{T}=2\nu(T)$.
\end{theorem}

\begin{theorem}[K\"{o}nig-Egerv\'ary] \label{konig}
	In any bipartite graph $G$, the number of edges in a maximum matching equals the number of vertices in a minimal vertex cover: $\nu(G)=\tau(G)$.
\end{theorem}

The remainder of this work is organized as follows. In Section \ref{minimum nullity} we give an algorithm to construct a tree, named $\mna{\mathbf{s}}$, from a given tree degree sequence $\mathbf{s}$. We prove, using the notion of matching number, that $\mna{\mathbf{s}}$ has minimum nullity among all trees with $\mathbf{s}$ as its degree sequence. In Section \ref{maximum nullity} we give an algorithm to construct a tree, named $\MNA{\mathbf{s}}$, from a given tree degree sequence $\mathbf{s}$. We prove, using the notion of matching number and annihilation number, that $\MNA{\mathbf{s}}$ has maximum nullity among all trees with $\mathbf{s}$ as its degree sequence. We also give a characterization of sequences with equal minimum and maximum nullity.

\section{Minimum Nullity}\label{minimum nullity}

The nullity of a graph is the nullity of its adjacency matrix. Let \(\mathbf{s}\) be a tree degree sequence of length $n$. The goal of this section is to prove that the minimum nullity among the nullity of all the trees in \(\mathcal{T}_{\mathbf{s}}\) is  $2l-n$, if $n-l \leq \lfloor \frac{n}{2} \rfloor$, and either $0$ or $1$, if $n-l > \lfloor \frac{n}{2} \rfloor$, where \(l\) is the number of number of vertices of degree 1 in \(\mathbf{s}\), we usually say that \(\mathbf{s}\) has \(l\) leaves instead to say it has \(l\) 1's.

\begin{proposition}\label{matchbound}
	Let $T$ be a tree of order $n> 2$ and $l$ be the number of leaves of $T$. If $M$ is a matching in $T$, then $|M|\leq \min\{n-l,\lfloor \frac{n}{2}\rfloor\}$.
\end{proposition}

\begin{proof}
	If $\min\{n-l,\lfloor \frac{n}{2}\rfloor\}=\lfloor \frac{n}{2}\rfloor$. Assume that $|M|>\lfloor \frac{n}{2}\rfloor$. Thus $\nu (T)>\lfloor \frac{n}{2}\rfloor$ and, by Theorem \ref{bevis}, $\rank{T}>2\lfloor \frac{n}{2}\rfloor$. If $n$ is even, then $\rank{T}>n$, which is absurd. If $n$ is odd, then $\rank{T}=n$. But this is impossible because, by Theorem \ref{bevis}, the rank of trees is always even. Therefore, $|M|\leq \lfloor \frac{n}{2}\rfloor$.
	
	If $n-l<\lfloor \frac{n}{2}\rfloor$. Assume that $|M|>n-l$. Since $n-l=|\{v\in V(T) : \deg{(v)}\geq 2 \}|$ and $M$ is a matching, by pigeon hole principle there exists an edge $vu\in M$ such that $v$ and $u$ are both leaves, which is impossible because $n>2$. Hence, $|M|\leq n-l$.
\end{proof}

The above proposition holds for tree degree sequences.

\begin{corollary}
	Let $\mathbf{s}$ be a tree degree sequence of length $n$ and $l$ be the number of 1's in $\mathbf{s}$. If $M$ is a matching in $T\in \mathcal{T}_{\mathbf{s}}$, then $|M|\leq \min\{n-l,\lfloor \frac{n}{2}\rfloor\}$.
\end{corollary}

 The nullity of any tree is strongly associated with the maximum matching structure of the tree, see \cite{jaume2018null}. In a tree degree sequence, a tree with mimimum (maximum) nullity, among all the trees in the tree sequence, is a tree with maximum (minimum) matching number, among the trees in the tree degree sequence.
 
 The next algorithm constructs a tree with the minimum nullity (maximum matching number) among all the trees with a given tree degree sequence $\mathbf{s}$ of length $n$. 
\begin{algorithm} \label{minalgo}
	Minimum nullity algorithm, $\mna{\mathbf{s}}$
	\begin{enumerate}[1)]
		\item INPUT. $\mathbf{s}:d_1,d_2,\ldots,d_n$
		\item  $V=\{v_1,v_2,\ldots,v_n\}$
		\item $l=|\{d_i\in \mathbf{s} : d_i=1\}|$
		\item $k=n-(d_{n}-1)$
		\item  $H(v_n)=\{v_1,v_{n-1},v_{n-2},\ldots,v_k\}$
		\item  $E=\{v_nv, \text{for } v\in H(v_n)\}$
		\item  $i=2$
		\item IF: $n-l\leq \lfloor \frac{n}{2} \rfloor$:
		\begin{enumerate}[8.1)]
			\item WHILE: $i\leq n-l$:
			\begin{enumerate}[8.1.1)]
				\item  $H(v_{n-i+1})=\{v_i,v_{k-1},v_{k-2},\ldots,v_{k-(d_{n-i+1}-2)}\}$
				\item  $E=E\cup \{v_{n-i+1}v, \text{for } v\in H(v_{n-i+1})\}$
				\item $k=k-(d_{n-i+1}-2)$
				\item $i= i+1$
			\end{enumerate}
			\item RETURN: $G(V,E)$
		\end{enumerate}
			\item WHILE: $i\leq l$
			\begin{enumerate}
				\item  $H(v_{n-i+1})=\{v_i,v_{k-1},v_{k-2},\ldots,v_{k-(d_{n-i+1}-2)}\}$
				\item $E=E\cup \{v_{n-i+1}v,\text{for }v\in H(v_{n-i+1})\}$
				\item $k=k-(d_{n-i+1}-2)$
				\item $i=i+1$
			\end{enumerate}
			\item $j=1$
			\item WHILE: $j<n-2l$
			\begin{enumerate}
				\item $A(v_{l+j})=\{v_{l+j+1}\}$
				\item $E=E\cup \{v_{l+j}v,\text{for }v\in A(v_{l+j})\}$
				\item $j=j+1$
			\end{enumerate}
			\item $E=E\cup \{v_1v_{n-l-1}\}$
			\item RETURN: $T(V,E)$
	\end{enumerate}
\end{algorithm}

Clearly the return of Algorithm \ref{minalgo} is a tree. From now on $\mna{\mathbf{s}}$ denote the tree $T(V,E)$, obtained by Algorithm \ref{minalgo} from a tree degree sequence $\mathbf{s}$. In Figure \ref{fig minalgo}, where in (a) the sequence satisfies $n-l\leq \lfloor \frac{n}{2} \rfloor$, and in (b) satisfies $n-l> \lfloor \frac{n}{2} \rfloor$, both cases for a sequence of order $n$.

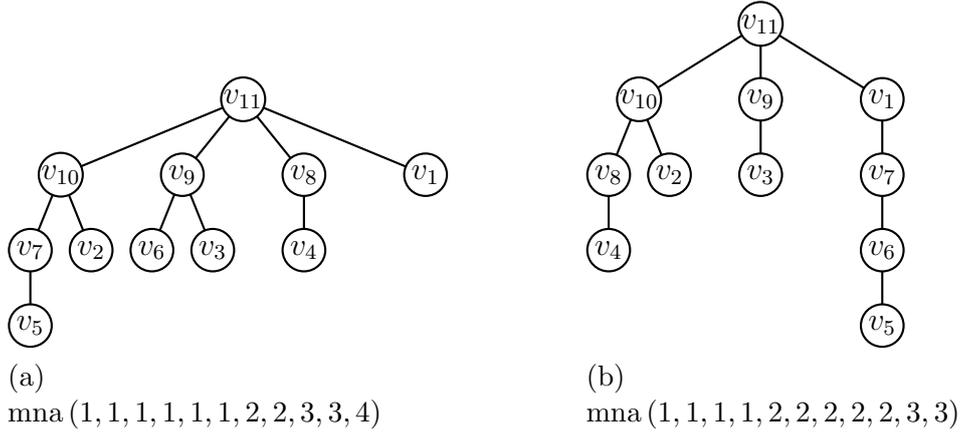
\begin{figure}[h!]
	\centering
	\begin{subfigure}[b]{0.3\textwidth}
		\begin{tikzpicture}[thick,scale=0.2]%
		\draw
		(12,15) node[]{$v_{11}$}
		(24,10) node[]{$v_1$} -- (12,15)
		(0,10) node{$v_{10}$} -- (12,15)
		(8,10) node{$v_9$} -- (12,15)
		(16,10) node{$v_8$} -- (12,15)
		(-2,5) node{$v_7$} -- (0,10)
		(2,5) node{$v_2$} -- (0,10)
		(6,5) node{$v_6$} -- (8,10)
		(10,5) node{$v_3$} -- (8,10)
		(16,5) node{$v_4$} -- (16,10)
		(-2,0) node{$v_5$} -- (-2,5);
		\end{tikzpicture}
			\caption{$\mna{1,1,1,1,1,1,2,2,3,3,4}$}
	\end{subfigure}
	\hspace{2cm}
	\begin{subfigure}[b]{0.3\textwidth}
		\begin{tikzpicture}[thick,scale=0.2]%
		\draw
		(8,15) node{$v_{11}$}
		(0,10) node{$v_{10}$} -- (8,15)
		(8,10) node{$v_9$} -- (8,15)
		(16,10) node{$v_1$} -- (8,15)
		(-2,5) node{$v_8$} -- (0,10)
		(2,5) node{$v_2$} -- (0,10)
		(8,5) node{$v_3$} -- (8,10)
		(-2,0) node{$v_4$} -- (-2,5)
		(16,5) node{$v_7$} -- (16,10)
		(16,0) node{$v_6$} -- (16,5)
		(16,-5) node{$v_5$} -- (16,0);
		\end{tikzpicture}
		\caption{$\mna{1,1,1,1,2,2,2,2,2,3,3}$}
	\end{subfigure}
	\caption{Example of the two cases of $\mna{\mathbf{s}}$}
	\label{fig minalgo}
\end{figure}

\begin{remark} \label{rem1}
	On one hand, when $n-l\leq \lfloor \frac{n}{2} \rfloor$, the vertices $v_i\in V(\mna{\mathbf{s}})$ with $1\leq i\leq n-l$, are all leaves. On the other hand, when $n-l> \lfloor \frac{n}{2} \rfloor$ the vertices $v_i\in V(\mna{\mathbf{s}})$ with $2\leq i\leq l+1$, are all leaves. 
\end{remark}



The objective is to prove that $\nu(\mna{\mathbf{s}})=\min\{n-l,\lfloor\frac{n}{2}\rfloor\}$ for every tree degree sequence \(\mathbf{s}\), where \(n\) is its length and \(l\) is its number of 1's (leaves). Since $\min\{n-l,\lfloor\frac{n}{2}\rfloor\}$ depends on the number of leaves in the sequence, we will split the prove into two cases: when $n-l\leq \lfloor\frac{n}{2}\rfloor$ and when $n-l>\lfloor\frac{n}{2}\rfloor$. 

Let \(\mathbf{s}\) be a tree degree sequence. Note that $l\geq \lceil\frac{n}{2}\rceil$ if and only if $n-l\leq \lfloor\frac{n}{2}\rfloor$. Therefore, $l<\lceil\frac{n}{2}\rceil$ if and only if $n-l>\lfloor\frac{n}{2}\rfloor$. 

Let $M_s=\{v_iv_j\in E(\mna{\mathbf{s}}):~1\leq i\leq n-l \text{ and } j=n-i+1\}$. Proposition \ref{numin 1} states that $M_s$ is a maximum matching in $\mna{\mathbf{s}}$ if \(l \geq \lceil \frac{n}{2} \rceil\).

\begin{proposition} \label{numin 1}
	Let $\mathbf{s}$ be a tree degree sequence of length $n>2$. If $l\geq \lceil \frac{n}{2}\rceil$, then $\nu{(\mna{\mathbf{s}})}=n-l$.
\end{proposition}

\begin{proof}
	On one hand, assume that there are two different edges in $M_s$ that are adjacent, say the edges $v_i v_{n-i+1}$ and $v_k v_{n-k+1}$, with $1\leq i\leq n-l$ and $1\leq k\leq n-l$. By Remark \ref{rem1} and that $n-l\leq \lfloor \frac{n}{2}\rfloor$, $v_i$ and $v_k$ are leaves. Thus, $v_{n-i+1}=v_{n-k+1}$ and $i=k$, which contradicts the fact that $v_i v_{n-i+1}$ and $v_k v_{n-k+1}$ are two different edges. Therefore, $M$ is a matching in $\mna{\mathbf{s}}$.
	
	On the other hand, by Proposition \ref{matchbound} and  since $l\geq \lceil \frac{n}{2}\rceil$, we have that $\nu(\mna{\mathbf{s}})\leq n-l$. Thus, $|M_s|$ is a maximum maximum matching in $\mna{\mathbf{s}}$ and $\nu (\mna{\mathbf{s}})=n-l$, because $|M_s|=n-l$.
\end{proof}
Using Proposition \ref{numin 1}, Theorem \ref{bevis}, and Theorem \ref{konig}, we obtain the nullity and the independence number of \(\mna{\mathbf{s}}\), when $l\geq \lceil \frac{n}{2}\rceil$.

\begin{corollary} \label{nullmin 1}
	Let $\mathbf{s}$ be a tree degree sequence of length $n>2$. If $l\geq \lceil \frac{n}{2}\rceil$, then $\nulidad{\mna{\mathbf{s}}}=2l-n$.
\end{corollary}


\begin{corollary}
	Let $\mathbf{s}$ be a tree degree sequence of length $n>2$. If $l\geq \lceil \frac{n}{2}\rceil$, then $\alpha({\mna{\mathbf{s}}})=l$.
\end{corollary}


When $l< \lceil \frac{n}{2}\rceil$, we have the following results.

\begin{proposition} \label{numin 2}
	Let $\mathbf{s}$ be a tree degree sequence of length $n>2$. If $l< \lceil \frac{n}{2}\rceil$, then \(\nu({\mna{\mathbf{s}}})=\left\lfloor\frac{n}{2}\right\rfloor.\)
\end{proposition}

\begin{proof}
	Let $S_1=\{v_1,v_2,\cdots,v_l,v_{n-l+1},\cdots,v_n\} \subset V(mna(s))$ and $S_2=V(T) - S_1$. Let $T_1$ and $T_2$ be the induced subgraphs of $\mna{\mathbf{s}}$ by $S_1$ and $S_2$, respectively. On one hand, since $\mna{\mathbf{s}}$ is a tree and, by construction, $T_1$ is a connected subgraph of $\mna{\mathbf{s}}$, we have that $T_1$ is a tree of order $n_1=2l$. On the other hand, $T_2$ is, by construction, a path of order $n_2=n-2l$. Therefore,  by Proposition \ref{numin 1}, we have that $\nu({T_1})=l$ and $\nu(T_2)=\lfloor \frac{n-2l}{2}\rfloor$ (because \(T_{2}\) is a path).

	Let $M_1$ and $M_2$ be maximum matchings in $T_1$ and $T_2$, respectively. Note that $M=M_1 \cup M_2$ is a matching in $T$, since $E(T_1)\cap E(T_2)=\emptyset$. Moreover, 
	\begin{eqnarray*}
	|M|&=&|M_1|+|M_2| \\
	&=&l+\left\lfloor \frac{n-2l}{2}\right\rfloor \\
	&=&\left\lfloor \frac{n}{2}\right\rfloor.
	\end{eqnarray*}
	But, since $l<\lceil\frac{n}{2}\rceil$ and $\mna{\mathbf{s}}=T_1\cup T_2 + v_1v_{l+1}$, by Proposition \ref{matchbound}, $M$ is a maximum matching in $\mna{\mathbf{s}}$ and $\nu(\mna{\mathbf{s}})=\left\lfloor \frac{n}{2}\right\rfloor$.
\end{proof}

Using Proposition \ref{numin 1}, Theorem \ref{bevis}, and Theorem \ref{konig}, we obtain the nullity and the independence number of \(\mna{\mathbf{s}}\), when $l < \lceil \frac{n}{2}\rceil$.
\begin{corollary}\label{nullmin 2}
	Let $\mathbf{s}$ be a tree degree sequence. If $l< \lceil \frac{n}{2}\rceil$, then \[\nulidad{\mna{\mathbf{s}}}=\left\{
	\begin{tabular}{c l}
	$1$&\text{if} n \text{is odd},\\
	$0$&\text{if} n \text{is even}.
	\end{tabular}
	\right.\].
\end{corollary}

\begin{corollary}
	\label{indmin 2}
	Let $\mathbf{s}$ be a tree degree sequence. If $l< \lceil \frac{n}{2}\rceil$, then \(\alpha({\mna{\mathbf{s}}})=\left\lceil \frac{n}{2} \right\rceil\).
\end{corollary}

Proposition~\ref{numin 1} and Proposition~\ref{numin 2}, together with their respective corollaries, states that for a given tree degree sequence $\mathbf{s}$ there is always a tree $T\in \mathcal{T}_{\mathbf{s}}$ such that $T$ has maximum matching number, minimum nullity and minimum independence number among the trees of $\mathcal{T}_{\mathbf{s}}$.  

\begin{definition}
	Let $\mathbf{s}$ be a tree degree sequence. The \textbf{maximum matching number of $\mathcal{T}_{\mathbf{s}}$}, denoted $\nu_M({\mathcal{T}_{\mathbf{s}}})$, is the maximum matching number among all trees in $ \mathcal{T}_{\mathbf{s}}$. The \textbf{minimum nullity of $\mathcal{T}_{\mathbf{s}}$}, denoted $\nulidadm{\mathcal{T}_{\mathbf{s}}}$, is the minimum nullity among all trees in $ \mathcal{T}_{\mathbf{s}}$. The \textbf{minimum independence number of $\mathcal{T}_{\mathbf{s}}$}, denoted $\alpha_m({\mathcal{T}_{\mathbf{s}}})$, is the minimum independence number among all trees in  $ \mathcal{T}_{\mathbf{s}}$.
\end{definition}

We collect all the previous results in a theorem.
\begin{theorem}\label{teominimum}
	If $\mathbf{s}$ is a tree degree sequence of length $n>2$, then \begin{align*}
	\nu_M(\mathcal{T}_{\mathbf{s}}) &=\left\{\begin{tabular}{c l}
	$n-l$&\text{if } $l\geq \lceil \frac{n}{2}\rceil$,\\
	$\lfloor \frac{n}{2}\rfloor$&\text{if } $l< \lceil \frac{n}{2}\rceil$,
	\end{tabular}
	\right.\\
	\nulidadm{\mathcal{T}_{\mathbf{s}}}&=\left\{
	\begin{tabular}{c l}
	$2l-n$&\text{if } $l\geq \lceil \frac{n}{2}\rceil$,\\
	1&\text{if } $l< \lceil \frac{n}{2}\rceil$ \text{and} $n$ \text{is odd},\\
	0&\text{if } $l< \lceil \frac{n}{2}\rceil$ \text{and} $n$ \text{is even,}
	\end{tabular}
	\right.
	\end{align*}
	and
	\begin{align*}
	\alpha_m(\mathcal{T}_{\mathbf{s}}) & =\left\{\begin{tabular}{c l}
	$l$&\text{if } $l\geq \lceil \frac{n}{2}\rceil$, \\
	$\lceil \frac{n}{2}\rceil$&\text{if } $l< \lceil \frac{n}{2}\rceil$. \phantom{\text{and} $n$ \text{ is even,}}
	\end{tabular}
	\right.
	\end{align*}
\end{theorem}

\section{Maximum Nullity}\label{maximum nullity}

In this section, given a tree degree sequence \(\mathbf{s}\) , we find a closed formula for the maximum nullity among the nullity of all trees in $\mathcal{T}_{\mathbf{s}}$, see Theorem \ref{teomaximum}. We explicitly find a tree with the maximum possible nullity, see Algorithm \ref{maxalgo}. We use the notion of annihilation number, which was introduced by Pepper in \cite{pepper2004binding}. This structural parameter is associated with the independence number for graph in general. If the graph is bipartite, it also associated with the matching number.

The next algorithm constructs a tree with the maximum nullity (minimum matching number) among all the trees with a given tree degree sequence $\mathbf{s}$ of length $n$.

\begin{algorithm}
	\label{maxalgo}
	Maximum nullity algorithm, $\MNA{\mathbf{s}}$
	\begin{enumerate}
		\item INPUT: $s:d_1,d_2,\ldots,d_n$
		\item  $V=\{v_1,v_2,\ldots,v_n\}$
		\item  $l=\sum_{d_i=1}1$
		\item  $\deg(v_n)=d_n$
		\item  $H(v_n)=\{v_1,v_2,\ldots,v_{d_n}\}$ 
		\item  $E=\{v_{n}v: v\in H(v_{n})\}$
		\item  $t=d_n+1$, $k=d_{n}$, $i=0$, $h=n$, $V_K=\empty$
		\item IF $t=n$:
		\begin{enumerate}
			\item  $i=\omega$
			\item RETURN $G(V,E)$, $V_K$, $\omega$
		\end{enumerate}
		\item  $i=i+1$
		\item WHILE $t<n$:
		\begin{enumerate}
			\item  $V_K=V_K\cup \{v_k\}$
			\item  $\deg(v_k)=d_{l+i}$ 
			\item  $H(v_k)=\{v_{h-1},v_{h-2},\ldots,v_{h-d_{l+i}+1}\}$
			\item  $E=E\cup \{v_{k}v: v\in H(v_{k})\}$
			\item  $t=t+d_{l+i}-1$
			\item IF $t=n$:
			\begin{enumerate}
				\item  $i=\omega$
				\item RETURN $G(V,E)$, $V_K$, $\omega$
			\end{enumerate}
			\item FOR $j\in ind(H(v_k))=\{f\in [n]:v_f\in H(v_k) \}$:
			\begin{enumerate}
				\item  $\deg(v_j)=d_j$
				\item  $H(v_j)=\{v_{k+1},v_{k+2},\ldots,v_{k+d_j-1}\}$
				\item  $E=E\cup \{v_{j}v: v\in H(v_{j})\}$
				\item  $t=t+d_{j}-1$
				\item IF $t=n$:
				\begin{enumerate}
					\item  $i=\omega$
					\item RETURN $G(V,E)$, $V_K$, $\omega$
				\end{enumerate}
				\item  $k=k+d_{j}-1$
			\end{enumerate}
			\item  $h=h-d_{l+i}+1$
			\item  $i=i+1$
		\end{enumerate}
	\end{enumerate}
\end{algorithm}

From now on we denote $\MNA{\mathbf{s}}$ to the tree obtained from Algorithm \ref{maxalgo}. In Figure \ref{fig maxalgo} we shown examples for each of the four possible execution forms of Algorithm \ref{maxalgo}.

\begin{figure}[h!] 
	\centering
	\begin{subfigure}[b]{0.3\textwidth}
		\begin{tikzpicture}[thick,scale=0.3]%
		\draw
		(0,0) node{$v_{9}$}
		(5,0) node{$v_1$} -- (0,0)
		(3.535533,3.535533) node{$v_2$} -- (0,0)
		(0,5) node{$v_3$} -- (0,0)
		(-3.535533,3.535533) node{$v_4$} -- (0,0)
		(-5,0) node{$v_5$} -- (0,0)
		(-3.535533,-3.535533) node{$v_6$} -- (0,0)
		(0,-5) node{$v_7$} -- (0,0)
		(3.535533,-3.535533) node{$v_8$} -- (0,0);
		\end{tikzpicture}
		\caption{$V_K=\emptyset$, $\omega=0$, $\MNA{1,1,1,1,1,1,1,1,8}$}
	\end{subfigure}
	\hspace{1cm}
	\begin{subfigure}[b]{0.3\textwidth}
		\begin{tikzpicture}[thick,scale=0.2]%
		\draw
		(0,0) node{$v_{15}$}
		(-5,0) node[]{$v_2$} -- (0,0)
		(0,5) node{$v_{1}$} -- (0,0)
		(0,-5) node{$v_3$} -- (0,0)
		(5,0) node{$v_4$} -- (0,0)
		(10,0) node{$v_{14}$} -- (5,0)
		(10,5) node{$v_{6}$} -- (10,0)
		(15,0) node{$v_5$} -- (10,0)
		(10,-5) node{$v_7$} -- (10,0)
		(5,-5) node{$v_{13}$} -- (10,-5)
		(10,-10) node{$v_{12}$} -- (10,-5)
		(15,-5) node{$v_{11}$} -- (10,-5)
		(20,-5) node{$v_8$} -- (15,-5)
		(15,-10) node{$v_9$} -- (15,-5)
		(20,0) node{$v_{10}$} -- (15,-5);
		\end{tikzpicture}
		\caption{$V_K=\{v_4,v_7\}$, $\omega=2$, $\MNA{1,1,1,1,1,1,1,1,1,1,2,4,4,4,4}$}
	\end{subfigure}	
	\hspace{1cm}
	\begin{subfigure}[b]{0.3\textwidth}
		\begin{tikzpicture}[thick,scale=0.2]%
		\draw
		(0,0) node{$v_{18}$}
		(-5,0) node[]{$v_2$} -- (0,0)
		(0,5) node{$v_{1}$} -- (0,0)
		(0,-5) node{$v_3$} -- (0,0)
		(5,0) node{$v_4$} -- (0,0)
		(5+3.535533,3.535533) node{$v_{17}$} -- (5,0)
		(5+3.535533,-3.535533) node{$v_{16}$} -- (5,0)
		(5+2*3.535533,2*3.535533) node{$v_5$} -- (5+3.535533,3.535533)
		(5+3.535533,3.535533+5) node{$v_6$} -- (5+3.535533,3.535533)
		(5,2*3.535533) node{$v_7$} -- (5+3.535533,3.535533)
		(5+3.535533+5,-3.535533) node{$v_9$} -- (5+3.535533,-3.535533)
		(5+3.535533,-3.535533-5) node{$v_8$} -- (5+3.535533,-3.535533)
		(10+2*3.535533,-2*3.535533) node{$v_{15}$} -- (10+3.535533,-3.535533)
		(10+2*3.535533,0) node{$v_{14}$} -- (10+3.535533,-3.535533)
		(15+2*3.535533,-2*3.535533) node{$v_{10}$} -- (10+2*3.535533,-2*3.535533)
		(10+2*3.535533,-2*3.535533-5) node{$v_{11}$} -- (10+2*3.535533,-2*3.535533)
		(10+2*3.535533,5) node{$v_{12}$} -- (10+2*3.535533,0)
		(15+2*3.535533,0) node{$v_{13}$} -- (10+2*3.535533,0)
		;
		\end{tikzpicture}
		\caption{\phantom{  } $V_K=\{v_4,v_9\}$, $\omega=2$, $\MNA{1,1,1,1,1,1,1,1,1,1,1,3,3,3,3,3,4,4}$}
	\end{subfigure}
	\caption{Examples of the three possible execution forms of Algorithm \ref{maxalgo}}
	\label{fig maxalgo}
\end{figure}
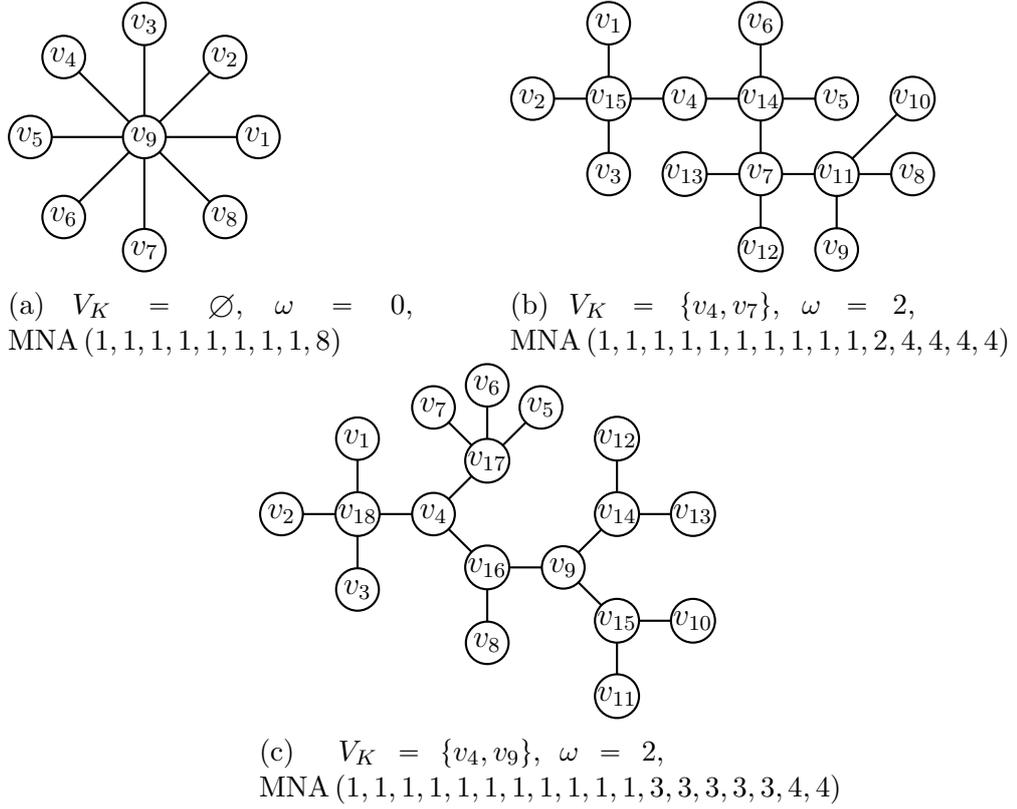


\begin{remark} \label{rem i}From the Algorithm \ref{maxalgo}, we can deduce que following statements about $\MNA{\mathbf{s}}$:
	\begin{itemize}
		\item[$\bullet$]$\MNA{\mathbf{s}}$ is a tree of order $n$.
		\item[$\bullet$]$\omega$ counts the number of times the WHILE loop runs, and for each time it runs adding a vertex to $V_K$. Thus $|V_K|=\omega$.
		\item[$\bullet$]Every vertex of $V_K$ is an internal vertex, i.e. they are vertices of degree 2 or more.
		\item[$\bullet$]The subindices of the vertices in $V_K$ are increasing and the distance between two consecutive vertices of $V_K$ is 2. Therefore, the distance between two vertices of $V_K$ is always even.
		\item[$\bullet$]Every internal vertex that is not a vertex of $V_K$ is adjacent to, at most, two vertices of $V_K$.
		\item[$\bullet$]Every internal vertex that is not a vertex of $V_K$ is adjacent to, at least, one leaf. 
		\item[$\bullet$]The only vertex of $V_K$ that can be neighbor of a leaf is the vertex with maximum subindex (the last vertex added to $V_K$). We call $v_{mk}$ to such vertex  and $l_{mk}$ to the number of leaves adjacent to $v_{mk}$.
		\item[$\bullet$] The distance from a leaf, which is not a neighbor of $v_{mk}$, to some vertex of $V_K$ is even.
		\item[$\bullet$] The vertex of $V_K$ with minimum subindex is $v_{l+1}$ (the first vertex added to $V_{K}$).
		\item[$\bullet$] Let $P_K$ be the path from $v_{l+1}$ to $v_{mk}$. Clearly $V(P_K)\cap V_K=V_K$.
	\end{itemize}
\end{remark}


\begin{lemma} \label{iedges}
	If $\mathbf{s}$ is a tree degree sequence of length $n$, then $$n-1-l=-l_{mk}+\sum_{v\in V_K}\deg(v),$$
	where \(l_{mk}\) is the number of leaves adjacent to \(v_{mk}\), see Algorithm \ref{maxalgo}.
\end{lemma}

\begin{proof} Given a vertex $v$, let $\deg_{\text{ext}}(v)$ denote the number of leaves adjacent to $v$ and let $\deg_{\text{int}}(v)$ denote the number of internal vertices adjacent to $v$. Thus,
	\begin{eqnarray}
	-l_{mk}+\sum_{v\in V_K}\deg(v)&=&-l_{mk}+\sum_{j=l+1}^{l+\omega}d_j \nonumber \\
	&=&-l_{mk}+d_{l+\omega}+\sum_{j=l+1}^{l+\omega-1}d_j \nonumber \\
	&=&-l_{mk}+\deg(v_{mk})+\sum_{j=l+1}^{l+\omega-1}d_j \nonumber \\
	&=&-l_{mk}+\deg_{\text{ext}}(v_{mk})+\deg_{\text{int}}(v_{mk})+\sum_{j=l+1}^{l+\omega-1}d_j \nonumber \\
	&=&\deg_{\text{int}}(v_{mk})+\sum_{j=l+1}^{l+\omega-1}d_j. \nonumber 
	\end{eqnarray}
	Clearly, by Remark \ref{rem i}, $\deg_{\text{int}}(v_{mk})+\sum_{j=l+1}^{l+\omega-1}d_j$ counts the number of internal edges (edges whose incident vertices are internal) that are incidents to some vertex of $V_K$; and by Remark \ref{rem i}, every internal edge is incident to one vertex of $V_K$. Thus, since we are in a tree \[\deg_{\text{int}}(v_{mk})+\sum_{j=l+1}^{l+\omega-1}d_j=n-l-1. \qedhere
	\]
\end{proof}

\begin{definition}[\cite{pepper2004binding}]
	For a graph $G$ with vertices $V= \{v_1, v_2, \ldots, v_n\}$, having degrees $d_i =
	\deg (v_i)$, with $d_1 \leq d_2 \leq \ldots \leq d_n$, and having $m$ edges, the \textbf{annihilation number}, $a = a(G)$, is
	defined to be the largest index such that $\sum_{i=1}^{a}d_i\leq m$.
\end{definition}

\begin{theorem}[\cite{pepper2004binding}] \label{pepper ind. bound}
	For any graph $G$, $\alpha(G)\leq a(G)$.
\end{theorem}

\begin{lemma} \label{nu lower bound}
	If $G$ is a bipartite graph of order $n$, then $\nu(G)\geq n-a(G)$.
\end{lemma}

\begin{proof}
	It is a direct consequence of Theorem \ref{pepper ind. bound} and Theorem \ref{konig}.
\end{proof}

\begin{theorem}
	If $T$ is a tree of order $n$, then $\nulidad{T}\leq 2a(T)-n$.
\end{theorem}

\begin{proof}
	It is a direct consequence of Lemma \ref{nu lower bound} and Theorem \ref{bevis}.
\end{proof}

Two graphs that have the same degree sequence have the same annihilation number. Hence, we can define the annihilation number of a degree sequence as $a(\mathbf{s})=a(G)$ for any $G\in \mathcal{G}_{\mathbf{s}}$. Usually, when the tree degree sequence \(\mathbf{s}\) is clear from the context, we just write \(a\) instead of \(a(\mathbf{s})\).

The following result gives a relation between $\omega$ from Algorithm \ref{maxalgo} and the annihilation number of the sequence $\mathbf{s}$.

\begin{lemma} \label{i}
	Let $\mathbf{s}$ be a tree degree sequence of length $n$. If $\omega$ is the output from Algorithm \ref{maxalgo} for $\mathbf{s}$, then $a(\mathbf{s})-l(\mathbf{s})\leq \omega \leq a(\mathbf{s})-l(\mathbf{s})+1$.
\end{lemma}

\begin{proof}
	From the proof of Lemma \ref{iedges} we have that: 
	\begin{eqnarray}
	\sum_{v\in V_K}\deg(v)&=&\sum_{j=l+1}^{l+\omega}d_j \nonumber \\
	&\geq&n-1-l \nonumber \\
	&\geq&\sum_{j=l+1}^{a}d_j , \nonumber
	\end{eqnarray}
	wich states that $\omega \geq a-l$. 
	
	Now, as $v_{mk}$ is the only vertex of $V_K$ that can be adjacent to a leaf, we have:
	\begin{eqnarray}
	\sum_{v\in (V_K-v_{mk})}\deg(v)&=&\sum_{j=l+1}^{l+\omega-1}d_j \nonumber \\
	&<&n-1-l \nonumber \\
	&<&\sum_{j=l+1}^{a+1}d_j, \nonumber
	\end{eqnarray}
	wich states that $\omega \leq a-l+1$. Hence \[a-l\leq \omega \leq a-l+1. \qedhere \] 
\end{proof}

As a consequence of the proof of Lemma \ref{i}, we have that $a-l=\omega$ if only if $l_{mk}=0$, and $a-l+1=\omega$ if only if $l_{mk}>0$.\\

Let $P_K$ be the path defined in Remark \ref{rem i}. Since \(P_K\) has even length, it has two maximum matchings. Let $M_K$ be a maximum matching of $P_K$ such that the edge incident to $v_{l+1}$ is in $M_K$ and the edge incident to $v_{mk}$ is not in $M_K$. Hence, every vertex of $V_K-v_{mk}$ is saturated by $M_K$ and $|M_K|=\omega-1$. Let $V_J=\{v\in V(\MNA{\mathbf{s}}):v\in N(V_K)-V(P_K)\text{ and } \deg(v)>1\}$. By Remark \ref{rem i}, every vertex of $V_J$ is neighbor of a leaf. Thus, $M_J=\{vu\in E(\MNA{\mathbf{s}}):v\in V_J \text{ and }u\in N(v) \text{ such that } \deg(u)=1 \}$ is a matching in \(\MNA{\mathbf{s}}\) such that \(M_{J} \cap E(P_{K})=\emptyset\). 

Note that $M_J \cup M_K$ is a matching in $\MNA{\mathbf{s}}$. If $u\in N(v_{mk})$ such that $\deg(u)=1$, then $M_J \cup M_K \cup uv_{mk}$ is also a matching in $\MNA{\mathbf{s}}$, due to the fact that $v_{mk}$ is unsaturated by $M_K$ and \(M_{J}\). We set 
\[
M_{\mathbf{s}}:=\left\{
\begin{array}{ll}
M_J \cup M_K \cup uv_{mk}, & \text{if } l_{mk}>0,\\
M_J \cup M_K & \text{otherwise},
\end{array}
\right.
\]
where \(u \in N(v_{mk})\cap l(\MNA{\mathbf{s}})\). An example of $M_{\mathbf{s}}$ is given in Figure \ref{fig Ms}, using the same trees from Figure \ref{fig maxalgo}.

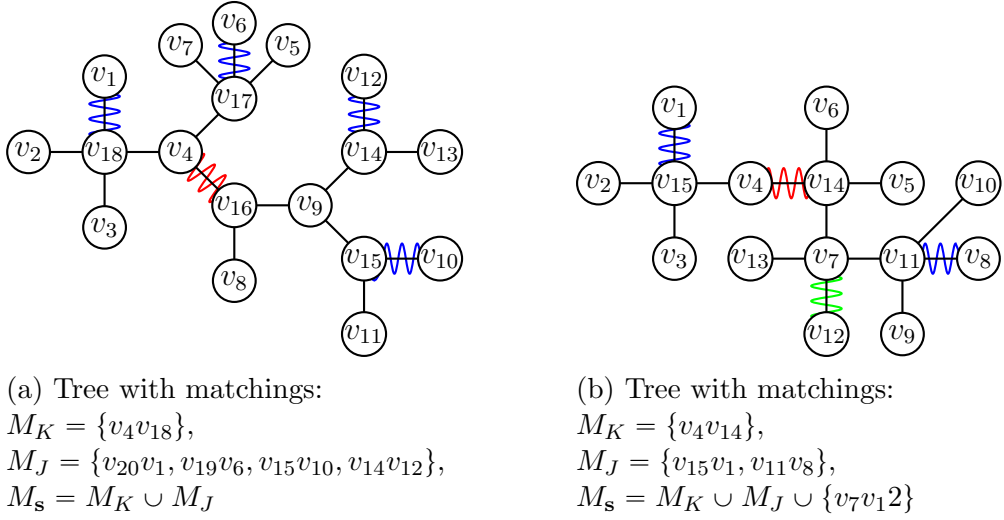
\begin{figure}[h!] 
	\centering
	\begin{subfigure}[b]{0.35\textwidth}
		\begin{tikzpicture}[thick,scale=0.2]%
		\draw[decorate, decoration={snake,amplitude=2mm,
			segment length=2mm},draw=red]
		(5+3.535533,-3.535533) -- (5,0);
		
		\draw[decorate, decoration={snake,amplitude=2mm,
			segment length=2mm},draw=blue]
		(0,5) -- (0,0)
		(6+3.535533,3.535533+5) -- (6+3.535533,3.535533)
		(11+2*3.535533,5) -- (11+2*3.535533,0)
		(15+2*3.535533,-2*3.535533-1) -- (10+2*3.535533,-2*3.535533-1);
		
		\draw
		(0,0) node{$v_{18}$}
		(-5,0) node[]{$v_2$} -- (0,0)
		(0,5) node{$v_{1}$} -- (0,0)
		(0,-5) node{$v_3$} -- (0,0)
		(5,0) node{$v_4$} -- (0,0)
		(5+3.535533,3.535533) node{$v_{17}$} -- (5,0)
		(5+3.535533,-3.535533) node{$v_{16}$} -- (5,0)
		(5+2*3.535533,2*3.535533) node{$v_5$} -- (5+3.535533,3.535533)
		(5+3.535533,3.535533+5) node{$v_6$} -- (5+3.535533,3.535533)
		(5,2*3.535533) node{$v_7$} -- (5+3.535533,3.535533)
		(5+3.535533+5,-3.535533) node{$v_9$} -- (5+3.535533,-3.535533)
		(5+3.535533,-3.535533-5) node{$v_8$} -- (5+3.535533,-3.535533)
		(10+2*3.535533,-2*3.535533) node{$v_{15}$} -- (10+3.535533,-3.535533)
		(10+2*3.535533,0) node{$v_{14}$} -- (10+3.535533,-3.535533)
		(15+2*3.535533,-2*3.535533) node{$v_{10}$} -- (10+2*3.535533,-2*3.535533)
		(10+2*3.535533,-2*3.535533-5) node{$v_{11}$} -- (10+2*3.535533,-2*3.535533)
		(10+2*3.535533,5) node{$v_{12}$} -- (10+2*3.535533,0)
		(15+2*3.535533,0) node{$v_{13}$} -- (10+2*3.535533,0)
		;
		\end{tikzpicture}
		\caption{Tree with matchings:\\ $M_K=\{v_4v_{18}\}$,\\ 
			$M_J=\{v_{20}v_1,v_{19}v_{6},v_{15}v_{10},v_{14}v_{12}\}$,\\ $M_{\mathbf{s}}=M_K\cup M_J$}
	\end{subfigure}	
	\hspace{1cm}
	\begin{subfigure}[b]{0.3\textwidth}
		\begin{tikzpicture}[thick,scale=0.2]%
		\draw[decorate, decoration={snake,amplitude=2mm,
			segment length=2mm},draw=red]
		(10,0) -- (5.5,0);
		
		\draw[decorate, decoration={snake,amplitude=2mm,
			segment length=2mm},draw=green]
		(10,-5) -- (10,-10);
		
		\draw[decorate, decoration={snake,amplitude=2mm,
			segment length=2mm},draw=blue]
		(0,0) -- (0,4.5)
		(15,-6) -- (19.5,-6);
		
		\draw
		(0,0) node{$v_{15}$}
		(-5,0) node[]{$v_2$} -- (0,0)
		(0,5) node{$v_{1}$} -- (0,0)
		(0,-5) node{$v_3$} -- (0,0)
		(5,0) node{$v_4$} -- (0,0)
		(10,0) node{$v_{14}$} -- (5,0)
		(10,5) node{$v_{6}$} -- (10,0)
		(15,0) node{$v_5$} -- (10,0)
		(10,-5) node{$v_7$} -- (10,0)
		(5,-5) node{$v_{13}$} -- (10,-5)
		(10,-10) node{$v_{12}$} -- (10,-5)
		(15,-5) node{$v_{11}$} -- (10,-5)
		(20,-5) node{$v_8$} -- (15,-5)
		(15,-10) node{$v_9$} -- (15,-5)
		(20,0) node{$v_{10}$} -- (15,-5);
		\end{tikzpicture}
		\caption{Tree with matchings:\\ $M_K=\{v_4v_{14}\},$\\ 
			$M_J=\{v_{15}v_1,v_{11}v_{8}\},$\\ $M_{\mathbf{s}}=M_K\cup M_J\cup\{v_7v_12\}$}
	\end{subfigure}
	\caption{The two cases of $M_{\mathbf{s}}$}
	\label{fig Ms}
\end{figure}

In the next proposition we prove, using the famous Berge's Lemma, that \(M_{\mathbf{s}}\) is a maximum matching of \(\MNA{\mathbf{s}}\).

\begin{lemma}[Berge's Lemma, see \cite{berge1957two}] \label{bergetheo}
	$M$ is a maximum matching if and only if there is no augmenting path relative to $M$.
\end{lemma}

\begin{proposition} \label{min nu}
	If $\mathbf{s}$ is a tree degree sequence of length $n$, then $M_{\mathbf{s}}$ is a maximum matching of $\MNA{\mathbf{s}}$. Moreover $\nu(\MNA{\mathbf{s}})=n-a(\mathbf{s})$.
\end{proposition}
\begin{proof}
	Note that if $v\in V(\MNA{\mathbf{s}})$ is unsaturated by $M_{\mathbf{s}}$, then $v$ is a leaf or $v=v_{mk}$. If $u$ and $v$ are two unsaturated vertices such that $\text{d}(u,v)$ is even, there is no augmenting path relative to $M$ from $u$ to $v$. In particular, by Remark \ref{rem i}, if $u$ and $v$ are two unsaturated leaf such that $u,v\notin N(v_{mk})$, then the distance from $u$ to $v$ is even, and therefore, there is no augmenting path relative to $M$ from $u$ to $v$. Hence, there are only two cases left to analyze: the path from a leaf to $v_{mk}$, when $v_{mk}$ is unsaturated, i.e. $a-l=\omega$; and the path between two leaves, one that is adjacent to $v_{mk}$ and one that is not, i.e. $a-l+1=\omega$.
	
	Case 1: Since $a-l=\omega$, by Remark \ref{rem i} the distance between \(v_{mk}\) and any leaf of \(\MNA{\mathbf{s}}\) is even. Thus, there is no augmenting path relative to $M_{\mathbf{s}}$ from a leaf to \(v_{mk}\). Therefore, by Berge's Lemma, $M_{\mathbf{s}}$ is a maximum matching of $\MNA{\mathbf{s}}$.
	
	By Remark \ref{rem i}, $M_{\mathbf{s}}$ saturates all neighbors of each vertex in $V_K$, and between two consecutive vertices of $V_K$ there is a saturated vertex, by Lemma \ref{iedges} we have that:
	\begin{eqnarray*}
	\nu(\MNA{\mathbf{s}})&=&|M_{\mathbf{s}}| \\
	&=&-(\omega-1)+\sum_{v\in V_K}\deg(v) \\
	&=&1-\omega+n-1-l \\
	&=&n-a.
	\end{eqnarray*}
	Case 2: $a-l+1=\omega$. Let $v$ and $u$ be two unsaturated leaves such that $v\notin N(v_{mk})$ and $u\in N(v_{mk})$. As $v_{mk}$ is matched to a leaf in $M_{\mathbf{s}}$, the path between \(v\) and \(u\) is not augmenting, because its two final edges are not in \(M_{\mathbf{s}}\).
	
	Consequently, there is no augmenting path relative to $M_{\mathbf{s}}$. Therefore, by Theorem  \ref{bergetheo}, $M_{\mathbf{s}}$ is a maximum matching of $\MNA{\mathbf{s}}$.
	
	As $M_{\mathbf{s}}$ saturates all non-leaves vertices that are neighbors of the vertices in $V_K$ and one leaf adjacent to $v_{mk}$, and between two consecutive vertices of $V_K$ there is a saturated vertex, by Lemma \ref{iedges} we have that:
	\begin{eqnarray*}
		\nu(\MNA{\mathbf{s}})&=&|M| \\
		&=&-(\omega-1)-(l(v_{mk})-1)+\sum_{v\in V_K}\deg(v) \\
		&=&2-\omega+n-1-l \\
		&=&n-a. 
	\end{eqnarray*} \qedhere
\end{proof}

A consequence of Lemma \ref{min nu}, and Lemma \ref{nu lower bound}, is that given a tree degree sequence $\mathbf{s}$, there is always a tree $T\in \mathcal{T}_{\mathbf{s}}$ such that has minimum matching number, and one of such tree is $\MNA{\mathbf{s}}$. 

\begin{corollary}
	If $\mathbf{s}$ is a tree  degree sequence of length $n$, then $$\nulidad{\MNA{\mathbf{s}}}=2a-n.$$
\end{corollary}

\begin{proof}
	A direct implication of Lemma \ref{min nu} and Theorem \ref{bevis}.
\end{proof}

\begin{corollary}
	If $\mathbf{s}$ is a tree  degree sequence of length $n$, then $$\alpha(\MNA{\mathbf{s}})=a.$$
\end{corollary}

Proposition~\ref{min nu}, together with their corollaries, states that $\MNA{\mathbf{s}}$ has minimum matching number, maximum nullity and maximum independence number among all the trees of $\mathcal{T}_{\mathbf{s}}$.  

Theorem~\ref{teomaximum} is the main result of this section, and a direct consequence of Proposition \ref{min nu} and their corollaries. But, in order to enunciate it properly, some definitions are necessary.

\begin{definition}
Let $\mathbf{s}$ be a tree degree sequence. The \textbf{minimum matching number of $\mathbf{\mathcal{T}_{\mathbf{s}}}$}, denoted $\nu_{\text{m}}(\mathcal{T}_{\mathbf{s}})$, is the maximum matching number among all trees $T$ in $\mathcal{T}_{\mathbf{s}}$. The \textbf{maximum nullity of $\mathbf{\mathcal{T}_{\mathbf{s}}}$}, denoted $\nulidadM{\mathcal{T}_{\mathbf{s}}}$, is the maximum nullity among all trees $T$ in $\mathcal{T}_{\mathbf{s}}$. The \textbf{maximum independence number of $\mathbf{\mathcal{T}_{\mathbf{s}}}$}, denoted $\alpha_{\text{M}}({\mathcal{T}_{\mathbf{s}})}$, is the minimum matching number among all trees $T$ in $\mathcal{T}_{\mathbf{s}}$.
\end{definition}

\begin{theorem}\label{teomaximum}
	If $\mathbf{s}$ is a degree sequence of trees of length $n$, then
	\begin{align*}
	\nu_{\text{m}}({\mathcal{T}_{\mathbf{s}})}&=n-l(\mathbf{s}),\\
	\nulidadM{\mathcal{T}_{\mathbf{s}}}&=2a(\mathbf{s})-n,
	\end{align*}
	and
	\begin{align*} 
	\alpha_{\text{M}}({\mathcal{T}_{\mathbf{s}})}&=a(\mathbf{s}).\phantom{-n}
	\end{align*}
\end{theorem}

Finally, the following results gives a characterization of the tree degree sequences with equal maximum and minimum matching number(nullity, independence number). 

\begin{theorem}
	Let $\mathbf{s}$ be a tree degree sequence of lenght $n>2$. Then  $\nu_{\text{m}}(\mathcal{T}_{\mathbf{s}})=\nu_{\text{M}}(\mathcal{T}_{\mathbf{s}})$ if only if $a(\mathbf{s})=l$ or $a(\mathbf{s})=\lfloor\frac{n}{2}\rfloor$.
\end{theorem}

\begin{proof} Since $\nu_{\text{M}}(\mathcal{T}_{\mathbf{s}})$ depend on the number of leaves of $\mathbf{s}$  and $\nu_{\text{m}}(\mathcal{T}_{\mathbf{s}})$ is always equal to $n-a$, it is easy to see that: when $l\geq \lceil\frac{n}{2}\rceil$
	\[\nu_{\text{m}}(\mathcal{T}_{\mathbf{s}})=\nu_{\text{M}}(\mathcal{T}_{\mathbf{s}})\iff a=l, \]\\
	and when $l< \lceil\frac{n}{2}\rceil$
	\[\nu_{\text{m}}(\mathcal{T}_{\mathbf{s}})=\nu_{\text{M}}(\mathcal{T}_{\mathbf{s}})\iff a=\left\lfloor\frac{n}{2}\right\rfloor.\]
\end{proof}

\begin{corollary}
	Let $\mathbf{s}$ be a tree degree sequence of lenght $n>2$. Then  $\nulidadm{\mathcal{T}_{\mathbf{s}}}=\nulidadM{\mathcal{T}_{\mathbf{s}}}$ if only if $a(\mathbf{s})=l$ or $a(\mathbf{s})=\lfloor\frac{n}{2}\rfloor$.
\end{corollary}

\begin{corollary}
	Let $\mathbf{s}$ be a tree degree sequence of lenght $n>2$. Then  $\alpha_{\text{m}}(\mathcal{T}_{\mathbf{s}})=\alpha_{\text{M}}(\mathcal{T}_{\mathbf{s}})$ if only if $a(\mathbf{s})=l$ or $a(\mathbf{s})=\lfloor\frac{n}{2}\rfloor$.
\end{corollary}

\begin{conjecture}
	If $\mathbf{s}$ is a tree degree sequence of lenght $n>2$, then for all positive integer $k$, with $\nu_m(\mathcal{T}_{\mathbf{s}})\leq k \leq \nu_M(\mathcal{T}_{\mathbf{s}})$, there exists a tree $T\in \mathcal{T}_{\mathbf{s}}$ such that $\nu(T)=k$.
\end{conjecture}

\section*{Acknowledgement}
The authors would like to thank to Vilmar Trevisan, Emilio Allem, Rodrigo Orsini Braga and Maikon Machado Toledo of Universidade Federal do Rio Grande do Sul for a wonderful time in our work's visit.

\section*{}
\textbf{Funding}: This work was partially supported by the Universidad Nacional de San Luis, grant PROIPRO 03-2216, and MATH AmSud, grant 18MATH-01. Dr. Daniel A. Jaume is funding by ``Programa de Becas de Integraci\'{o}n Regional para argentinos'', grant 2075/2017.

\section*{References}

\bibliographystyle{plain}

\bibliography{TAGcitas}

\tikzstyle{every node}=[circle, draw, fill=white!,
inner sep=0.1pt, minimum width=11pt]

\end{document}